\documentclass{amsart}
\usepackage{amssymb,latexsym,amsthm}
\usepackage[centertags]{amsmath}
\usepackage{amsfonts}
\usepackage{graphicx}
\newcommand{\abs}[1]{\left\lvert{#1}\right\rvert}
\newcommand{\norm}[1]{\left\|{#1}\right\|}
\newcommand{\gen}[1]{\left\langle{#1}\right\rangle}

\DeclareMathOperator{\pr}{\rm{pr}}

\DeclareMathOperator{\inter}{\rm{int}}
\DeclareMathOperator{\bd}{\partial}
\DeclareMathOperator{\cl}{cl}

\DeclareMathOperator{\per}{\rm{Per}}
\DeclareMathOperator{\diff}{\rm{Diff}}

\DeclareMathOperator{\IFS}{\rm{IFS}}

\newcommand{\mc}{\mathcal}

\newcommand{\ol}{\overline}
\newcommand{\til}{\tilde}

\newcommand{\R}{\mathbb{R}}\newcommand{\N}{\mathbb{N}}
\newcommand{\Z}{\mathbb{Z}}
\newcommand{\T}{\mathbb{T}}

\newcommand{\sm}{\setminus}

\newcommand{\cf}{c.f.\ }
\newcommand{\ie}{i.e.\ }
\newcommand{\eg}{e.g.\ }

\newtheorem{theorem}{Theorem}%[theorem] %[section]
\newtheorem{corollary}[theorem]{Corollary}
\newtheorem{lemma}[theorem]{Lemma}
\newtheorem{proposition}[theorem]{Proposition}

\newtheorem{question}[theorem]{Question}

\newtheorem*{theorem*}{Theorem}

\theoremstyle{definition}
\newtheorem{definition}[theorem]{Definition}

\theoremstyle{remark}
\newtheorem{remark}[theorem]{Remark}

\title[Transitivity of iterated function systems]{Transitivity of generic semigroups of area-preserving surface diffeomorphisms}
\author{Andres Koropecki}
\address{Universidade Federal Fluminense, Instituto de Matem\'atica, Rua M\'ario Santos Braga S/N, 24020-140 Niteroi, RJ, Brasil}
\email{koro@mat.uff.br}

\author{Meysam Nassiri}
\address{Instituto Nacional de Matem\'atica Pura e Aplicada, Estrada
  Dona Castorina, 110. 22460-320, Rio de Janeiro, Brasil. }
\address{Current address: Department of Mathematics, Institute for Advanced Studies in Basic Sciences (IASBS)
P. O. Box 45195-1159 Zanjan, Iran}
\email{nassiri@impa.br}
\thanks{The authors were partially supported by CNPq-Brasil}

\begin{document}
\begin{abstract}
We prove that the action of the semigroup generated by a $C^r$ generic pair of area-preserving diffeomorphisms of a compact orientable surface is transitive.
\end{abstract}

\maketitle

\section{Introduction} \label{sec:intro}

In this article we consider a compact orientable surface $S$ with a smooth volume form $\omega$, and we denote by $\diff^r_\omega(S)$ the space of $C^r$ diffeomorphisms from $S$ to itself that leave $\omega$ invariant (\ie such that $f^*(\omega)=\omega$), endowed with the $C^r$ topology. 
Our main result is the following

\begin{theorem}\label{th:main} There is a residual set $\mc{R}\subset \diff^r_\omega(S)\times \diff^r_\omega(S)$ $(r\in \N \cup \{\infty\})$ with the product $C^r$ topology such that if $(f,g)\in \mc{R}$ then the iterated function system $\IFS(f,g)$ is transitive.
\end{theorem}
Saying that $\IFS(f_0,f_1)$ is transitive means that there is a point $x$ and a sequence $\{i_k\}$ of $1$'s and $0$'s such that $$\{f_{i_k}f_{i_{k-1}}\cdots f_{i_1}(x) : k\in \N\}$$ is dense in $S$, and it is equivalent to say that the action of the semigroup generated by $\{f_0,f_1\}$ is transitive. 

Theorem \ref{th:main} generalizes a result of Moeckel \cite{moeckel}, where he proves drift in the annulus for $\IFS(f,g)$, where $f$ is a monotone twist map and $g$ is a generic diffeomorphism. This is also related to the work of Le Calvez \cite{lecalvez-drift}. The most important difference in our case is that we do not require the twist condition, and combining the drift in annular invariant regions with generic arguments, we are able to obtain transitivity in the whole surface.

The dynamics of groups or semigroups of diffeomorphisms, besides its own
importance, is a useful tool for understanding the dynamics of certain
(single) diffeomorphisms. In fact, the topic of this paper is motivated by
several open problems in conservative and Hamiltonian dynamics.
One of them is the instability problem (or the so-called Arnold diffusion)
in Hamiltonian dynamics.  It has been conjectured \cite{arnold} that for $C^r$
generic perturbations of a high-dimensional integrable Hamiltonian system
(or symplectomorphism), many orbits drift between KAM invariant tori. This
problem remains open in its full generality, although many partial (but deep)
results have been obtained in the last decades (\cf \cite{CY, DLS, Ma, MS, xia2} and references there).

Another important problem is the topological version of the Pugh-Shub
conjecture \cite{pugh-shub}, which says that $C^r$ generic partially hyperbolic symplectic (or conservative) diffeomorphisms of a compact manifold are robustly transitive.

An approach to deal with these problems is using additional structure (\eg partial hyperbolicity of certain sets, skew-product structure, invariant foliations) to obtain an associated iterated function system which reflects certain properties of the original system (\cf \cite{moeckel, tesis-meysam, np}). In that setting, results like Theorem \ref{th:main} can be helpful to obtain transitivity or drift in the original system.

For instance, for symplectic diffeomorphisms which are the product of an Anosov diffeomorphism and a conservative surface diffeomorphism close to an integrable system it is possible to prove (using Theorem \ref{th:main}) that $C^r$-generic (symplectic) perturbations are transitive. Such applications of Theorem \ref{th:main} on transitivity of certain partially
hyperbolic sets or diffeomorphisms  will be treated elsewhere.

We remark that these problems have different nature in the $C^1$ topology. Bonatti and Crovisier \cite{bonatti-crovisier}, using a $C^1$ closing lemma, proved that $C^1$ generic conservative diffeomorphisms are transitive (see also \cite{abc}).
However, if we consider surface diffeomorphisms, this result is clearly false if we use the $C^r$ topology with $r$ sufficiently large, due to the KAM phenomenon. In particular, the number of diffeomorphisms in Theorem \ref{th:main} is optimal if $r$ is large (for instance, $r\geq 16$ \cite{douady}), in the sense that just one generic diffeomorphism is generally not transitive.

The proof of Theorem \ref{th:main} relies on the following result, which seems interesting by itself: for a $C^r$-generic $f$ ($r\geq 16$), the invariant \emph{frontiers} of $f$ are pairwise disjoint (see \S\ref{sec:frontiers} for a precise definition of frontier and Lemma \ref{lem:generic-disjoint} for a precise statement). The proof of this result is strictly two-dimensional, and that  restricts our main theorem to dimension two. However, we conjecture that the conclusion of Theorem \ref{th:main} also holds in higher dimensions. 

The following relevant questions were motivated by the present work.
\begin{question} For a $C^r$-generic pair of area-preserving diffeomorphisms, does any of the following properties hold?
\begin{enumerate}
\item $\IFS(f,g)$ is robustly transitive (\ie if for small perturbations of $f$ and $g$ the iterated function system is still transitive);
\item $\IFS(f,g)$ is ergodic (\ie $\{f,g\}$-invariant sets have measure $0$ or $1$);
\item $\IFS(f,g)$ is minimal (\ie there are no compact $\{f,g\}$-invariant sets other than the whole surface and the empty set).
\end{enumerate}
\end{question}

Let us say a few words about the proof of Theorem \ref{th:main}. We divide it in several sections. In \S\ref{sec:ifs} we provide several definitions of transitivity of group actions and iterated function systems, and we prove that they are all equivalent in the area-preserving setting. This simplifies the proof of Theorem \ref{th:main}. 
In \S\ref{sec:koro} we state a result from \cite{koro} (\cf Theorem \ref{th:continuo}) which gives a good description of aperiodic invariant continua in the area-preserving setting , and plays a crucial role in the proof of our main theorem. When the surface is $\T^2$ or $S^2$, it is possible to prove the theorem without these results, using additional arguments, but \cite{koro} allows to unify these arguments to prove the theorem for arbitrary surfaces.

In \S\ref{sec:generic} we introduce generic conditions and we state a result of Mather which relates open invariant sets with aperiodic invariant continua.
The main idea in the proof of Theorem \ref{th:main} is to reduce the problem to proving that for generic $(f,g)$, there are no ``nice annular continua'' periodic by both $f$ and $g$. We call these nice continua \emph{frontiers}; in \S\ref{sec:frontiers} we define them and prove some elementary facts about these sets. 
In \S\ref{sec:nondense} we prove that $\IFS(f,g)$ is transitive if there are no frontiers which are periodic for both $f$ and $g$ simultaneously. Using Theorem \ref{th:continuo}, we prove that all periodic frontiers of $f$ are pairwise disjoint if $f$ satisfies certain generic properties. 
Taking both $f$ and $g$ generic we are left with the problem of separating, by means of a perturbation of $g$, the family of periodic frontiers of $f$ from the corresponding family of $g$. To do this, we extend to our setting the arguments used by Moeckel in \cite{moeckel} where he proves a similar result for twist maps. This is done in \S\ref{sec:moeckel}. Finally, in \S\ref{sec:proof} we complete the proof of Theorem \ref{th:main}.

\subsection*{Acknowledgments}
We are grateful to J. Franks, P. Le Calvez, and E. Pujals for insightful discussions. We also thank the anonymous referee for the corrections and suggestions.

\section{Transitivity of group actions and iterated function systems}
\label{sec:ifs}

Throughout this section we assume that $M$ is an arbitrary manifold, not necessarily compact, and $\mc{F}$ is a family of homeomorphisms from $M$ to itself. We denote by $\gen{\mc{F}}^+$ and $\gen{\mc{F}}$ the semigroup and the group generated by $\mc{F}$, respectively. For $x\in M$, we write the orbits of the actions of this semigroup and group as
$$\gen{\mc{F}}^+(x) = \{f(x): f\in \gen{\mc{F}}^+\}\quad \text{ and }\quad \gen{\mc{F}}(x) = \{f(x) : f\in \gen{\mc{F}}\}.$$ 
We say that a set $E\subset M$ is $\mc{F}$-invariant if $f(E)=E$ for each $f\in \mc{F}$. In this case, it is clear that $f(E)=E$ for each $f\in \gen{\mc{F}}$.
If there is a point $x\in M$ such that $\gen{\mc{F}}(x)$ is dense in $M$, we say that $\gen{\mc{F}}$ is transitive. Similarly, if $\gen{\mc{F}}^+(x)$ is dense for some $x$ we say that $\gen{\mc{F}}^+$ is transitive.

A sequence $\{x_n : n\in \N \text{ (resp. $n\in \Z$})\}$ is called a branch (resp. full branch) of an orbit of $\IFS(\mc{F})$ if for each $n\in \N$ (resp. $n\in \Z$) there is $f_n\in \mc{F}$ such that $x_{n+1}=f_n(x_n)$. Here, $\IFS(\mc{F})$ stands for the iterated function system associated to $\mc{F}$. We say that $\IFS(\mc{F})$ is transitive if there is a branch of an orbit which is dense in $M$. 

The following lemmas show that if the maps in question preserve a 
finite measure with total support, then all the different notions of transitivity are equivalent.

\begin{lemma} Let $\mu$ be a finite Borel measure on $M$ with total support, and let $\mc{F}$ be a family of homeomorphisms from $M$, all of which leave $\mu$ invariant. Let $U$ and $V$ be open subsets of $M$ such that $f(U)\cap V\neq \emptyset$ for some $f\in \gen{\mc{F}}$. Then there is $\hat{f}\in \gen{\mc{F}}^+$ such that $\hat{f}(U)\cap V\neq \emptyset$.
\end{lemma}
\begin{proof}
Any $f\in \gen{\mc{F}}$ has the form $f=f_nf_{n-1}\cdots f_1$ for some $f_i$ such that one of $f_i$ or $f_i^{-1}$ belong to $\mc{F}$. We prove the result by induction on $n$. If $n=1$, then either $f\in \mc{F}\subset \gen{\mc{F}}^+$ (and there is nothing to do) or $f^{-1}\in \mc{F}$. In this case, since $f^{-1}$ preserves $\mu$ and the open set $f(U)\cap V$ has positive measure, there is $k\geq 2$ such that $$f^{-k}(f(U)\cap V)\cap f(U)\cap V\neq \emptyset.$$
In particular, letting $\hat{f}= f^{-k+1}$, we have that $\hat{f}\in \gen{\mc{F}}^+$ and $\hat{f}(U)\cap V \neq \emptyset$ as required.

Now suppose the proposition holds for some fixed $n\geq 1$, and suppose $f(U)\cap V\neq \emptyset$ where $f=f_{n+1}f_n\cdots f_1$ and $f_i\in \mc{F}$ or $f_i^{-1}\in \mc{F}$ for each $i$. Then $f_n\cdots f_1(U)\cap f_{n+1}^{-1}(V)\neq \emptyset$. By induction hypothesis (with $f_{n+1}^{-1}(V)$ instead of $V$) we find that there is $g\in \gen{\mc{F}}^+$ such that $g(U)\cap f_{n+1}^{-1}(V)\neq \emptyset$. Thus, $f_{n+1}(g(U))\cap V\neq \emptyset$. Applying the case $n=1$ with $g(U)$ instead of $U$, we see that there is $h\in \gen{\mc{F}}^+$ such that $h(g(U))\cap V\neq \emptyset$. Hence, $\hat{f}=hg\in\gen{\mc{F}}^+$ satisfies the required condition.
\end{proof}

We will only use $(4)\implies (8)$ from the next lemma, but we state the other equivalences for the sake of completeness.
\begin{lemma} \label{lem:trans-ifs} If $\mc{F}$ is a \emph{countable} family of homeomorphisms from $M$ to itself preserving a finite Borel measure with compact support, then the following are equivalent:
\begin{enumerate}
\item $\gen{\mc{F}}$ is transitive;
\item For any pair of nonempty open sets $U$ and $V$ of $M$, there is $f\in \gen{\mc{F}}$ such that $f(U)\cap V\neq \emptyset$;
\item There is a residual set $\mc{R}\subset M$ such that $\cl{\gen{\mc{F}}(x)}=M$ for each $x\in \mc{R}$;
\item $\gen{\mc{F}}^+$ is transitive;
\item For any pair of nonempty open sets $U$ and $V$ of $M$, there is $f\in \gen{\mc{F}}^+$ such that $f(U)\cap V\neq \emptyset$;
\item There is a residual set $\mc{R}^+\subset M$ such that $\cl{\gen{\mc{F}}^+(x)}=M$ for each $x\in \mc{R}^+$;
\item There is a full branch of an orbit of $\IFS(\mc{F})$ which is dense in $M$;
\item $\IFS(\mc{F})$ is transtive;
\item There is a residual subset $\mc{R}'\subset M$ such that for each $x\in \mc{R}'$ there is a branch of an orbit of $\IFS(\mc{F})$ starting at $x$ which is dense in $M$.
\end{enumerate}
\end{lemma}
\begin{proof}
Note that (2) $\implies$ (5) because of the previous lemma. To prove that (5) $\implies$ (9), consider a countable basis of open sets $\{V_n:n\in \N\}$ of $M$, and let $R_n$ be the set of all $x\in M$ such that $f(x)\in V_n$ for some $f\in \gen{\mc{F}^+}$. Clearly,  $R_n$  is open, and it is dense by (5). Thus $\mc{R}=\cap_n R_n$ is a residual set such that $\gen{\mc{F}}^+(x)$ is dense in $M$ for each $x\in \mc{R}$. Now 
$\mc{R}'=\bigcap_{f\in \gen{\mc{F}}} f(\mc{R})$
is a residual $\mc{F}$-invariant set, and for every $x\in \mc{R}'$ we can find a sequence $f_n\in \gen{\mc{F}}^+$ such that $f_n\cdots f_1(x)\in V_n$, so that $\{f_n\cdots f_1(x):n\in \N\}$ is a dense branch of an orbit starting at $x$.
It is clear that (9) implies all the other conditions, and all the other conditions imply $(2)$, so we are done.
\end{proof}

\section{Invariant continua}
\label{sec:koro}

By a continuum, we mean a compact connected set.

\begin{definition} \label{def:annular}
As in \cite{koro}, we say that a continuum $K$ is annular if it has an annular neighborhood $A$ such that $A\sm{K}$ has exactly two connected components, both homeomorphic to annuli. We call any such $A$ an $\emph{annular neighborhood}$ of $K$. 
\end{definition}

This definition is equivalent to saying that $K$ is the intersection of a sequence $\{A_i\}$ of closed topological annuli such that $A_{i+1}$ is an essential subset of $A_i$ (\ie it separates the two boundary components of $A_i$), for each $i\in \N$. 

Recall that $\Omega(f)$ denotes the nonwandering set of $f$, that is, the set of points $x\in S$ such that for each neighborhood $U$ of $x$ there is $n>0$ such that $f^n(U)\cap U\neq\emptyset$. We will need the following result:

\begin{theorem}[\cite{koro}] \label{th:continuo} Let $f\colon S\to S$ be a homeomorphism of a compact orientable surface such that $\Omega(f)=S$. If $K$ is an $f$-invariant continuum, then one of the following holds:
\begin{enumerate}
\item $f$ has a periodic point in $K$;
\item $K$ is annular;
\item $K=S=\T^2$ (a torus);
\end{enumerate}
\end{theorem}

\begin{remark} If $f$ is area-preserving, then $\Omega(f)=S$, so the first hypothesis of the above theorem always holds.
\end{remark}

\begin{remark} Theorem \ref{th:continuo} implies that if $f$ is a $C^r$-generic area-preserving diffeomorphism and $K$ is an aperiodic invariant continuum (\ie one which contains no periodic points of $f$), then $K$ is annular. This follows from the well known fact that a $C^r$-generic surface diffeomorphism has periodic points (see for example \cite[Corollary 2]{zanata} and \cite[\S8]{xia-area}).
\end{remark}

\begin{remark}\label{rem:wandering} 
We will use the following observation several times: 
If $\Omega(f)=S$ (in particular, if $f$ is area-preserving) and $\{U_i\}_{i\in\N}$ is a family of pairwise disjoint open sets which are permuted by $f$ (\eg the connected components of the complement of a compact periodic set) then each $U_i$ is periodic for $f$.
\end{remark}

\section{Annular continua and frontiers}
\label{sec:frontiers}

The following definitions and observations will be important to obtain disjoint families of invariant continua in the next sections.

\begin{definition} \label{def:frontier} Let $K$ be an annular continuum with annular neighborhood $A$, and denote by $\bd A^-$ and $\bd A^+$ the two boundary components of $A$. Then  $A\sm K$  has exactly two (annular) components, which we denote by $K^-$ and $K^+$, where $K^-$ is the one whose closure contains $\bd^-A$. We denote by $\bd^-K$ and $\bd^+K$ the boundaries of $K^-$ and $K^+$ in $A$, respectively. If $K=\bd^-K=\bd^+K$, we say that $K$ is a \emph{frontier}.
\end{definition}

\begin{lemma} \label{lem:frontier} If $K$ is an annular continuum with empty interior, then $K$ contains a unique frontier. 
\end{lemma}
\begin{proof} 
If $U_+=\inter\ol{K^+}$ and $U_-=\inter\ol{K^-}$, then $U_+\cup U_-$ is open and dense in $A$. Since $U_+$ and $U_-$ are disjoint and since $\inter \ol {U^+}=U^+$ (and similarly for $U^-$), it follows that $\bd U_+=\bd U_-$. Let  $\alpha=\bd U_+=\bd U_-$. It is clear that $\alpha$ is a frontier and $\alpha\subset K$. 

Now suppose that $\beta\subset K$ is a frontier. Since $K^-\subset \beta^-$ and  $K$ has empty interior, $\beta^-\sm K^-$ has empty interior  as well. Thus, $U_- = \inter\ol{K^-}=\inter\ol{\beta^-}=\beta^-$. Therefore, $\alpha=\bd U_-=\bd \beta^-= \beta$. This completes the proof.
\end{proof}

\begin{corollary} Let $f\colon S\to S$ be a homeomorphism and $K\subset S$ an $f$-invariant annular continuum with empty interior. Then the unique frontier $\hat{K}\subset K$ is $f$-invariant.
\end{corollary}
\begin{proof} It is clear that $f(\hat{K})\subset K$ is a frontier, so $f(\hat{K})=\hat{K}$ by uniqueness.
\end{proof}

\begin{lemma}\label{lem:frontier-disjoint}
Let $f\colon S\to S$ be a homeomorphism such that $\Omega(f)=S$, and $K_1$, $K_2$ two periodic frontiers containing 
no periodic point of $f$. Then either $K_1\cap K_2=\emptyset$, or $K_1=K_2$.
\end{lemma}
\begin{proof} By considering an appropriate power of $f$ instead of $f$, we may assume that $K_1$ and $K_2$ are $f$-invariant. If $K_1\cap K_2\neq \emptyset$, then $K=K_1\cup K_2$ is connected, has empty interior, and contains 
no periodic point of $f$, so by Theorem \ref{th:continuo} it is an annular continuum. By Lemma \ref{lem:frontier} there is a unique frontier $\hat{K}$ contained in $K$; hence $K_1=K_2=\hat{K}$.
\end{proof}

\section{Generic conditions}
\label{sec:generic}
From now on, $S$ is a compact orientable surface and $\omega$ is a smooth area element on $S$.

The main generic properties that we will require are comprised in the following
\begin{definition} We say that $f\in \diff^r_\omega(S)$ is \emph{Moser generic} if the following conditions are met:
\begin{enumerate}
\item All periodic points of $f$ are either elliptic or hyperbolic (saddles);
\item There are no saddle connections between hyperbolic periodic points (\ie the stable and unstable manifold of any pair of hyperbolic periodic points are transverse);
\item The elliptic periodic points are Moser stable. This means that if $f^n(p)=p$ and $p$ is elliptic then
\begin{itemize}
\item there are arbitrarily small periodic curves surrounding $p$;
\item the restriction of $f^n$ to each of these curves is an irrational rotation, and all the rotation numbers are different.
\end{itemize}
\end{enumerate}
\end{definition}

\begin{remark} \label{remark-moser}
Conditions (1) and (2) are $C^r$-generic for any $r\geq 1$, due to Robinson \cite{robinson}. These are usually referred to as the Kupka-Smale property for area-preserving diffeomorphisms. Condition (3) is $C^r$-generic due to KAM theory if $r$ is large enough (for instance, $r\geq 16$ \cite{douady}).
Thus, Moser generic diffeomorphisms are $C^r$ generic for $r\geq 16$.
Also note that if $f$ is Moser generic, then $f^n$ is Moser generic for any $n\neq 0$.
\end{remark}

%\war

In \cite{mather-area} Mather studies the prime ends compactification of open sets invariant by an area-preserving diffeomorphism.
Recall that an open set $U\subset S$ is a \emph{residual domain} if it is a connected component of $S\sm K$ for some continuum $K$, where $K$ is not a single point.
By \cite[Lemma 2.3]{mather-area}, a residual domain has finitely many boundary components. 
We will also use the following result from that article.

%We will use the following result,   is a direct consequence of the corollary to Theorem 5.1 and Lemma 2.3 from that article. We say that an open set $U\subset S$ is a \emph{residual domain} if it is a connected component of $S\sm K$ for some continuum $K$, where $K$ is not a single point.

\begin{theorem}[\cite{mather-area}] \label{th:mather-boundary} Let $f\in \diff^r_\omega(S)$ be a Moser generic diffeomorphism and $p$ a hyperbolic periodic point. Then any two branches of $p$ have the same closure.
\end{theorem}

By a branch of $p$ one means a connected component of $W^s_f(p)\sm{p}$ or $W^u_f(p)\sm {p}$). Thus if we define $K_{p,f} = \ol{W^s_f(p)}$, we have that $K_{p,f}$ is the closure of any stable or unstable branch of $p$.

\begin{proposition} \label{pro:pixton} For a Moser generic $f\in \diff^r_\omega(S)$, if $p$ is a hyperbolic periodic point and $K_{p,f}$ has a planar neighborhood, then $p$ has a homoclinic intersection.
\end{proposition}
\begin{proof}
The proof of \cite{pixton} assumes that the whole manifold is planar, but it works almost without modification if instead we assume that the stable and unstable manifolds of $p$ are contained in a planar submanifold of $S$. 

This is clear from the proof presented in \cite{oliveira-torus} (the first part of the proof of Theorem 2, p. 582, dealing only with the sphere) since it is a purely topological argument taking place in a neighborhood of the union of the stable and unstable manifolds of $p$, and the only part where the dynamics is relevant is in a local argument in a neighborhood of $p$. In our setting, using the fact that $K_{p,f}$ is the closure of both the stable and unstable manifolds of $p$, we can restrict our attention to a planar neighborhood of $p$, and by embedding this neighborhood in a sphere the same proof applies.
\end{proof}

\begin{definition}\label{def:res} For a Moser generic $f$, we denote by $\mathcal{G}_f$ the set of all Moser generic $g\in \diff^r_\omega(S)$ such that the following property holds: for any hyperbolic periodic point $p$ of $f$ and $k>0$,
\begin{equation}\label{eq:*}
\text{if }K_{p,f}\cap g^k(K_{p,f})\neq \emptyset \text{ then } W^s_f(p)\pitchfork g^k(W^s_f(p))\neq \emptyset \neq W^u_f(p)\pitchfork g^k(W^s_f(p))\quad
\end{equation}
 where $A \pitchfork B$ stands for the set of transversal intersections of $A$ and $B$.
\end{definition}

\begin{proposition}\label{pro:transversal} If $f$ is Moser generic, then $\mathcal{G}_f$ is $C^r$-residual.
\end{proposition}
\begin{proof}

Fix a Moser generic $g\in \diff^r_\omega(S)$, $k\in \N$, and a hyperbolic periodic point $p$ of $f$. Since the set $\per_k(g)$ of periodic points of $g$ with period at most $k$ is finite, using a perturbation of $g$ of the form $hgh^{-1}$ with $h$ close to the identity we may assume that $\per_k(g)\cap \bd K_{p,f} =\emptyset$ (because $\bd K_{p,f}$ has empty interior). Suppose $K_{p,f}$ intersects $g^k(K_{p,f})$. Since $g$ preserves area and $K_{p,f}$ is connected, this implies that $g^k(K_{p,f})\cap \bd K_{p,f}\neq \emptyset$. Choose $x\in g^k(K_{p,f})\cap \bd K_{p,f}$, and note that $x\notin \per_k(g)$. 
Let $x_n\in W^s_f(p)$ and $y_n\in W^s_f(p)$ be such that $x_n\to x$ and $g^k(y_n)\to x$. Let $U$ be a neighborhood of $x$ such that $g^{i}(U)$ is disjoint from $U$ for $-k\leq i \leq 0$. By standard arguments we may choose, for any sufficiently large $n$, a map $\til{h}\in \diff^r_\omega(S)$ supported in $g^{-1}(U)$ and $C^r$-close to the identity such that $\til{h}(g^{k-1}(y_n))=g^{-1}(x_n)$. Letting $\til{g}=g\til{h}$ we obtain a map $C^r$-close to $g$ such that $\til{g}^k(y_n)=x_n\in W^s_f(p)$. Thus $\til{g}^k(W^s_f(p))\cap W^s_f(p)\neq \emptyset$, and this intersection can be made transverse with a perturbation of $\til{g}$. Note that once the intersection is transverse, it persists under new perturbations. Since $K_{p,f}$ is the closure of $W^u_f(p)$ as well, and $\tilde{g}^k(K_{p,f})\cap K_{p,f}\neq \emptyset$, we may use the same argument to obtain a perturbation $\hat{g}$ of $\til{g}$ which also has a point of transversal intersection between $W^u_f(p)$ and $\hat{g}^k(W^s_f(p))$. This shows that, for $p$ and $k$ fixed, the set of maps $g$ satisfying (\ref{eq:*}) is dense in $\diff^r_\omega(S)$. Since condition (\ref{eq:*}) is also open, the set $\mathcal{U}_{p,k}\subset \diff^r_\omega(S)$ where property (\ref{eq:*}) holds for this choice of $p$ and $k$ is open and dense. Since $\mathcal{G}_f$ is the intersection of the sets $\mathcal{U}_{p,k}$ over all $k\in \N$ and all hyperbolic periodic points $p$ of $f$ (which are countably many), this completes the proof.
\end{proof}

\begin{proposition}\label{pro:homo} Let $f$ be Moser generic, $U$ a periodic connected open set and $K\subset \bd U$ a periodic nontrivial continuum. If $K$ has a planar neighborhood, then $K$ contains no periodic points.
\end{proposition}
\begin{proof}
Let $n$ be such that $f^n(K)=K$ and $f^n(U)=U$. Suppose that there is a periodic point $p$ of $f$ in $K$. We assume that $f^n(p)=p$, by increasing $n$ if necessary. Since $f$ (and so $f^n$) is Moser generic, $p$ is either elliptic (Moser stable) or hyperbolic. By an argument of Mather, we can show that $p$ is hyperbolic: suppose on the contrary that $p$ is Moser stable. Then there is a sequence of invariant disks converging to $p$ such that the dynamics of $f^n$ on the boundary of each disk is minimal. If $D$ is such a disk and $K$ intersects $\bd D$, then $\bd D\subset K\subset \bd U$, which implies by connectedness that $U$ is either contained in $D$ or disjoint form $D$. The first case is not possible if we choose $D$ small enough, while the second case contradicts the fact that $p\in \bd U$. Thus $p$ is hyperbolic.

Since $K$ has a planar neighborhood, it follows from Proposition \ref{pro:pixton} that $p$ has a homoclinic intersection. By the $\lambda$-lemma, this implies that there is a decreasing sequence of rectangles $R_i$ with boundary in $W^s_f(p)\cup W^u_f(p)\subset K$ accumulating on $p$. This contradicts the connectedness of $U$: in fact $U$ cannot be contained in all of the rectangles $R_i$, but by connectedness if $U$ is not contained in $R_i$ then $\ol{U}$ is disjoint from the interior of $R_{i+1}$, which contains points of $\bd{U}$. This is a contradiction.
\end{proof}

\begin{proposition}\label{pro:main} Let $f$ be Moser generic, and let $U$ be an $f$-periodic residual domain. If $U$ is also $g$-periodic for $g\in \mathcal{G}_f$, then $\bd U$ has no periodic points of $f$.
\end{proposition}

\begin{proof} If $K$ is a boundary component of $U$, then there is $n>0$ such that $f^n(K)=K=g^n(K)$. Suppose that there is a periodic point $p$ of $f$ in $K$. We assume that $f^n(p)=p$ and $f^n(U)=U=g^n(U)$ by increasing $n$ if necessary. Since $f$ is Moser generic, the first paragraph of the proof of Proposition \ref{pro:homo} shows that $p$ must be hyperbolic. 

By \cite[Proposition 11.1]{mather-area}, it follows that the stable and unstable manifolds of $p$ are contained in $K$, so that $K_p\doteq K_{p,f}$ is contained in $K$.  We will now show that $K_p$ has a planar neighborhood. Assume by contradiction that $K_p$ has no planar neighborhood, and consider two cases.

Suppose first that $K_p, g(K_p), g^2(K_p),\dots$ are pairwise disjoint. Then, for any $m>0$ we can find $m$ disjoint nonplanar manifolds with boundary $N_1,\dots N_m$ in $S$, by taking disjoint neighborhoods of $K_p, g(K_p), \dots, g^{m-1}(K_p)$. Moreover, we may assume that no component of $\ol{S\sm N_i}$ is a disk, by adding disks to $N_i$ if necessary. This implies that each component of $\ol{S\sm {\cup_i} N_i}$ has Euler characteristic at most $0$ (for it has at least two boundary components). On the other hand, the Euler characteristic of each $N_i$ is strictly negative (because by our assumption they are nonplanar and have nonempty boundary), so we have 
$$2-2g=\chi(S) = \chi(N_1)+\cdots +\chi(N_m) + \chi\left(\ol{S\sm \cup_i N_i}\right)\leq -m,$$
where $g$ is the genus of $S$. Since this can be done for all $m$, we get a contradiction. 

Now suppose that $g^k(K_p)\cap K_p\neq \emptyset$ for some $k>0$. Then by Proposition \ref{pro:transversal}, we can find arcs $\gamma_s,\gamma_u\subset g^k(W^s_f(p))\subset K$ such that $\gamma_s$ has a nonempty transversal intersection with $W^s_f(p)$, and $\gamma_u$ has a nonempty transversal intersection with $W^u_f(p)$. By a standard $\lambda$-lemma argument, this implies that, arbitrarily close to $p$, we can find a decreasing sequence of arbitrarily small rectangles $R_i$ bounded by arcs in $W^s_f(p)$, $W^u_f(p)$, $f^{n_ik}(\gamma_u)$ and $f^{-n_ik}(\gamma_s)$ (for some $n_i\in \Z$) accumulating on $p$. Since all of these arcs are contained in $K\subset \bd U$, the same argument used in the proof of Proposition \ref{pro:homo} produces a contradiction from the connectedness of $U$. This completes the proof that $K_p$ has a planar neighborhood.

 But if $K_p$ has a planar neighborhood, we know from Proposition \ref{pro:homo} that $f$ cannot have periodic points on $K_p$, contradicting the definition of $K_p$. This completes the proof.
\end{proof}

\section{Periodic frontiers from invariant open non-dense sets}
\label{sec:nondense}

\begin{lemma} 
Let $f$ be Moser generic and $g\in \mathcal{G}_f$. Suppose that there is an open nonempty $\{f, g\}$-invariant set $U\subset S$ which is not dense in $S$. Then there is a continuum $K\subset S$ with empty interior and $n>0$ such that $f^n(K)=g^n(K)=K$ and there are no periodic points of $f$ or $g$ in $K$. 
\end{lemma}

\begin{proof} 
Let $U_0$ be a connected component of $U$. Since $f$ and $g$ are  area-preserving, $U_0$ must be periodic for each of them.
Let $V$ be a connected component of $S\sm \ol{U}_0$. Then $V$ is also periodic for both $f$ and $g$, and it is a residual domain, so Proposition \ref{pro:main} applied to some power of $f$ and some power of $g$ implies that there are finitely many boundary components of $V$, none of which contains a periodic point of $f$ or $g$. If $K$ is one such component, since there are finitely many components, it must be periodic for both $f$ and $g$, so there is $n$ such that $f^n(K)=g^n(K)=K$ as required.
\end{proof}

The characterization of transitivity of $\IFS(f,g)$ used to prove Theorem \ref{th:main} is the following:

\begin{lemma}\label{lem:transitive} If $f$ is Moser generic and $g\in \mathcal{G}_f$, then $\IFS(f,g)$ is transitive if there is no frontier $K\subset S$ which is periodic for both $f$ and $g$.
\end{lemma}
\begin{proof}
Suppose that $\IFS(f,g)$ is not transitive. Then by Lemma \ref{lem:trans-ifs}, there is an open nonempty set $U\subset S$ such that $$\gen{f,g}(U)=\bigcup_{h\in \gen{f,g}} h(U)$$
is not dense in $S$. The above set is open, nonempty and $\gen{f,g}$-invariant (hence, $f$-invariant and $g$-invariant). Thus by the previous lemma there is a continuum $K$ which has no periodic points such that $f^n(K)=g^n(K)=K$. By Theorem \ref{th:continuo}, $K$ must be annular. By Lemma \ref{lem:frontier} and its corollary, $K$ contains a unique frontier which is both $f^n$-invariant and $g^n$-invariant. This completes the proof.
\end{proof}

We will need the following result:
\begin{lemma} \label{lem:generic-disjoint} If $f\in \diff^r_\omega(S)$ is Moser generic, then the invariant frontiers of $f$ are aperiodic and pairwise disjoint.  
\end{lemma}
\begin{proof} If $K$ is an invariant frontier, then $K$ has at most two residual domains, and $K$ is the union of their boundaries. Since obviously any boundary component of these residual domains consists of more than one point and the frontiers have planar neighborhoods, it follows from Proposition \ref{pro:homo} that $K$ is aperiodic. Now our claim follows from Lemma \ref{lem:frontier-disjoint}.
\end{proof}

\section{Separating two families of disjoint continua}
\label{sec:moeckel}

Let $Q=[0,1]^2$ denote the unit square. 
We say that a continuum $K\subset Q$ \emph{crosses the square} horizontally (resp. vertically) if the two horizontal (resp. vertical) sides of $Q$ are contained in different connected components of $Q\sm K$.

The main result of this section (Lemma \ref{lem:separa-cont}) is very similar to a result of Moeckel \cite{moeckel} for families of Lipschitz graphs, which is in turn inspired by the embedding theory in \cite{embedology}. 
It says that for any pair of disjoint families of continua horizontally crossing the unit square, we can find a smooth, area-preserving map $C^\infty$-close to the identity, coinciding with the identity in a neighborhood of $\bd Q$, such that no element of the first family is mapped to an element of the second one.

First we state a simple lemma, which is proved in \cite{moeckel}.

\begin{lemma}\label{lem:separa} Let $A$ and $A'$ be two subsets of $\R^3$ with upper box dimensions at most $1$. Then for almost every $(x,y,z)\in \R^3$, the sets $A+(x,y,z)$ and $A'$ are disjoint.
\end{lemma}

%We will also need this lemma (which is essentially included in Moeckel's article).
We will also need the following result which is proved, although not explicitly stated, in \cite{moeckel}:

\begin{lemma}\label{lem:boxdim} Let $S\subset \R^3$ be a set that is linearly ordered by the partial ordering defined as $(x_1,x_2,x_3) \preceq (y_1,y_2,y_3)$ if $x_i\leq y_i$ for each $i$. Then the upper box dimension of $S$ is at most $1$.
\end{lemma}
\begin{proof}
Consider the map $\phi\colon \R^3\to \R$ defined by $\phi(x_1,x_2,x_3)=x_1+x_2+x_3$. Note that $\phi$ is order-preserving in $S$; in fact, if $x,y\in S$ and $x\prec y$ (\ie $x\preceq y$  and $x\neq y$), we have that all three coordinates of $y-x$ are nonnegative, and at least one is positive, so that $0< \phi(y-x)=\phi(y)-\phi(x)$. 
In particular, $\phi|_S$ is a bijection onto its image $E=\phi(S)$. Denoting by $\psi\colon E\to S$ its inverse, we have that if $x=(x_1,x_2,x_3)\prec (y_1,y_2,y_3)=y$, then denoting by $\norm{\cdot}_1$ the $l^1$ norm, $$\norm{\psi(\phi(y))-\psi(\phi(x))}_1 = \norm{y-x}_1= \abs{y_1-x_1} + \abs{y_2-x_2}+\abs{y_3-x_3} =  \abs{\phi(y)-\phi(x)},$$ because $y_i-x_i\geq0$ for all $i\in \{1,2,3\}$. This means that $\psi$ is an isometry if we use the $l^1$ norm in $\R^3$. Since the upper box dimension is independent of the norm used to compute it (and $E\subset \R$ implies that $E$ has upper box dimension at most $1$), the claim follows.
\end{proof}

\begin{lemma}\label{lem:separa-cont} Let $\mc{K}$ and $\mc{K}'$ be two families of pairwise disjoint continua horizontally crossing the unit square $Q=[0,1]^2$ which are disjoint from $[0,1]\times [0,\delta]$ and $[0,1]\times [1-\delta,1]$ for some small $\delta>0$. Then there exists an area-preserving $C^\infty$ diffeomorphism $h\colon Q\to Q$ arbitrarily close to the identity in the $C^\infty$ topology, such that $h$ coincides with the identity in a neighborhood of $\bd Q$ and  $h(\mc{K})=\{h(K):K\in \mc{K}\}$ is disjoint from $\mc{K}'$.
\end{lemma}

\begin{proof} Denote by $\pr_2\colon Q\to [0,1]$ the projection onto the second coordinate. Given $K_1, K_2 \in \mc{K}$, and $x\in [0,1]$, we write $K_1\prec_x K_2$ if 
$$\max \pr_2(K_1\cap \ell) < \max \pr_2(K_2\cap \ell),$$
where $\ell=\{x\}\times [0,1]$. It is easily verified that $\prec_x$ is a total ordering of $\mc{K}$. Moreover, this ordering is independent of $x$. In fact, $\ell\sm K_1$ has exactly one connected component containing $(x,1)$, which is included in $\ell\cap K_1^+$, where $K_1^+$ is the connected component of $Q\sm K_1$ containing $[0,1]\times \{1\}$. If $K_1\prec_x K_2$ then $K_2\cap K_1^+\neq \emptyset$, and since $K_2$ is connected and $K_1$ is disjoint from $K_2$ it follows that $K_2\subset K_1^+$. This in turn implies that $K_1 \prec_{x'} K_2$ for any $x'\in [0,1]$. Thus we may unambiguously write $K_1\prec K_2$ if $K_1\prec_{x} K_2$ for some $x$. 

\begin{figure}[ht!]
\centering{\resizebox{0.4\textwidth}{!}{\includegraphics{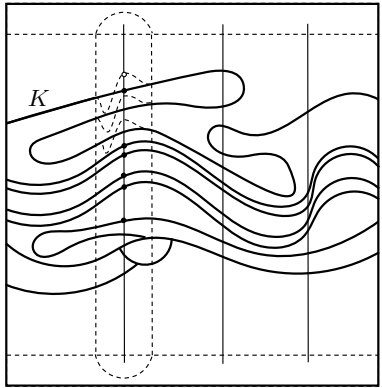}}}
  \caption{Construction of $h$} \label{fig2}
\end{figure}

Fix three numbers $x_i\in (0,1)$ such that $x_1<x_2<x_3$, and define $\ell_i = \{x_i\}\times [\delta,1-\delta]$, $i\in \{1,2,3\}$. Note that for each $K\in \mc{K}$, the line $\ell_i$ intersects $K$. Thus, for each $i$ we may define a map $$y_i\colon \mc{K}\to \R, \quad y_i(K)= \max\pr_2(K\cap \ell_i).$$
%Fix three numbers $t_i\in [0,1]$ such that $0<t_1<t_2<t_3$, and define $\ell_i = \{t_i\}\times [a-\delta,b+\delta]$, $i\in \{1,2,3\}$. For each $i$ we define a map $$y_i\colon \mc{K}\to \R, \quad y_i(K)= \max\pr_2(K\cap \ell_i).$$

From the previous observations, $y_i(K_1)<y_i(K_2)$ for some $i$ if and only if $y_i(K_1)<y_i(K_2)$ for all $i\in \{1,2,3\}$ (\ie iff $K_1\prec K_2$). 
Thus the hypotheses of Lemma \ref{lem:boxdim} hold for the set $$A=\{(y_1(K),y_2(K),y_3(K)): K\in \mc{K}\} \subset [0,1]^3,$$ so we know that the upper box dimension of $A$ is at most one.

We can order the elements of $\mc{K}'$ in a similar way and define maps $y_i'\colon \mc{K}'\to \R$ in an analogous way, obtaining a set  $$A'=\{(y_1'(K),y_2'(K),y_3'(K)): K\in \mc{K}'\}$$ 
which also has upper box dimension at most one.
It follows from Lemma \ref{lem:separa} that we can choose
%$(t_1,t_2,t_3)\subset \R^3$ 
$(t_1,t_2,t_3)\subset [0,1]^3$ 
arbitrarily close to the origin such that $A+(t_1,t_2,t_3)$ is disjoint from $A'$.

For each $i$, we take small disjoint rectangular neighborhoods $R_i$ of each $\ell_i$, and we can define $C^\infty$ functions
 $\alpha_i\colon Q\to \R$ such that $\alpha_i(x,y)=x$ for $(x,y)$ in a small neighborhood of $\ell_i$, $\alpha_i(x,y)=0$ for $(x,y)\in Q\sm R_i$, and $\abs{\alpha_i(x,y)}\leq 1$ for all $(x,y)\in Q$. It is easy to see that the Hamiltonian flow $\phi_i^t$ induced by the vector field $(-\partial_y\alpha_i, \partial_x \alpha_i)$ has the property that 
 $\phi_i^t(Q) =Q$ for $t\in \R$, 
 $\phi_i^t(x,y)=(x,y)$ if $(x,y)\in Q\sm R_i$, and there is 
 $\epsilon>0$ such that $\phi_i^t(x,y) = (x,y)+(0,t)$ if $\abs{t}<\epsilon$ and 
  $(x,y)\in \ell_i$.

We choose $(t_1,t_2,t_3)\in \R^3$ with $\abs{t_i}< \epsilon$ for each $i$, such that $A+(t_1,t_2,t_3)$ is disjoint from $A'$. Note that the maps $\phi_1^{t_1}$, $\phi_2^{t_2}$ and $\phi_3^{t_3}$ have disjoint supports. Letting $h=\phi_1^{t_1}\circ \phi_2^{t_2} \circ \phi_3^{t_3}$, we have that $h(A)=A+(t_1,t_2,t_3)$, which is disjoint from $A'$. 
Thus, if $K\in \mc{K}$ and $K'\in \mc{K'}$ then $h(K)\neq K'$, since $\max \pr_2 h(K)\cap \ell_i\neq \max \pr_2 K\cap \ell_i$ for some $i\in \{1,2,3\}$. If $\epsilon$ is chosen sufficiently small, we get $h$ arbitrarily $C^\infty$ close to the identity. This completes the proof.
\end{proof}

\section{Proof of Theorem \ref{th:main}}
\label{sec:proof}

First we state a finer result, and we show how it implies Theorem \ref{th:main}. Then, after introducing a useful definition, we complete the proof.

\begin{theorem}\label{th:transitive} For $r\geq 16$, if $f\in \diff^r_\omega(S)$ is Moser generic, there is a $C^r$-residual set $\mc{R}_f\subset \diff^r_\omega(S)$ such that if $g\in \mc{R}_f$ then $\IFS(f,g)$ is transitive.
\end{theorem}

\begin{proof}[Proof of Theorem \ref{th:main} assuming Theorem \ref{th:transitive}]
 It is easy to see that the set of pairs $(f,g)$ such that $\IFS(f,g)$ is transitive is a $G_\delta$ set in $\diff^r_\omega(S)\times \diff^r_\omega(S)$ with the product $C^0$ topology (and hence, also with the product $C^r$ topology, for any $r\in \N\cup\{\infty\}$). From Theorem \ref{th:transitive} and Remark \ref{remark-moser}, this set is dense if $r\geq 16$. To conclude the proof, it is enough to note that $\diff^r_\omega(S)$ is dense in $\diff^k_\omega(S)$ if $r\geq k$ (see \cite{zehnder}). 
\end{proof}

\begin{definition} \label{def:essential} We say that a continuum $K$ is \emph{essential} in an open annular set $A$ if $K$ is annular and $A$ is an annular neighborhood of $K$ (see Definition \ref{def:annular}).
If $E$ is a closed annulus, we say that $K$ is essential in $E$ if $K$ is essential in some annular neighborhood of $E$.
\end{definition}

\begin{proof}[Proof of Theorem \ref{th:transitive}] Let $E\subset S$ be diffeomorphic to a closed annulus, $n>0$, and $f\in \diff^r_\omega(S)$ be Moser generic. Let $\mc{R}^n_{f,E}$ be the set of all $g\in \diff^r_\omega(S)$ such that there is no frontier which is essential in $E$ and invariant by both $f^n$ and $g^n$. We will prove that $\mc{R}^n_{f,E}$ is open and dense.

Let $A$ be an open annular neighborhood of $E$. Denote by $\mc{K}(f^n)$ the family of all $f^n$-invariant frontiers which are essential in $E$. To prove density, we will find $h\in \diff^\infty_\omega(S)$, arbitrarily $C^r$-close to the identity, such that $\mc{K}(hg^nh^{-1})\cap \mc{K}(f^n)=\emptyset$.

Denote by $A_1$ and $A_2$ the two (annular) components of $A\sm E$. 
Let $\phi\colon Q=[0,1]^2\to S$ be a $C^\infty$ embedding such that 
%$\phi(\{0\}\times[0,1])\subset A_1$, $\phi(\{1\}\times[0,1])\subset A_2$
$\phi([0,1]\times\{0\})\subset A_1$, $\phi([0,1]\times\{1\})\subset A_2$ and $\phi(Q)\subset A$ (this can easily be obtained using a small tubular neighborhood of a simple arc joining the boundary components of $E$). By a result of Moser \cite{moser}, we may assume that $\phi$ maps the area element of $Q$ to the restriction of $\omega$ to $\phi(Q)$. 
Let $$\hat{\mc{K}}_0(f^n)=\{\phi^{-1}(K\cap \phi(Q)):K\in \mc{K}(f^n)\},$$
and denote by $\hat{\mc{K}}(f^n)$ the family of all connected components of $\hat{\mc{K}}_0(f^n)$ that separate $Q$. Since elements of $\mc{K}(f^n)$ are pairwise disjoint (by Lemma \ref{lem:generic-disjoint}), $\hat{\mc{K}}(f^n)$ is a family of pairwise disjoint continua. It is easy to see that each of these continua horizontally separates $Q$, and they satisfy the hypotheses of Lemma \ref{lem:separa-cont}. 

If $g$ is another Moser generic element of $\diff^r_\omega(S)$, then we can define $\mc{K}(g^n)$ and $\hat{\mc{K}}(g^n)$ similarly. By Lemma $\ref{lem:separa-cont}$, there is $h_0\in \diff^\infty_{\rm{Leb}}(Q)$ arbitrarily $C^\infty$ close to the identity, which coincides with the identity in a neighborhood of $\bd Q$ and such that $h_0(\hat{\mc{K}}(g^n))$ and $\hat{\mc{K}}(f^n)$ are disjoint. 
 %Defining $\til{h}\in \diff^\infty_\omega(S)$ by $h(x)=h_0(\phi^{-1}(x))$ 
Defining $h\in \diff^\infty_\omega(S)$ by $h(x)=\phi(h_0(\phi^{-1}(x)))$ 
if $x\in \phi(Q)$ and $h(x)=x$ otherwise, we have that $h(\mc{K}(g^n))$ and $\mc{K}(f^n)$ are disjoint, and $$\mc{K}((hgh^{-1})^n)=\mc{K}(hg^nh^{-1})=h(\mc{K}(g^n)).$$ Since we may assume that $h$ is arbitrarily $C^r$ close to the identity, we may also assume that $\til{g}=hgh^{-1}$  is $C^r$ close to $g$, and by construction $\til{g}^n$ and $f^n$ have no common invariant frontiers which are essential in $E$.
This proves the $C^r$ density of $\mc{R}^n_{f,E}$.

We now prove that the complement of $\mc{R}^n_{f,E}$ is $C^0$ (thus $C^r$) closed. Let $\{g_n\}$ be a sequence of diffeomorphisms not in $\mc{R}^n_{f,E}$, such that $g_n\to g\in \diff^r_\omega(S)$ in the $C^0$ topology. Then, for each $m$ there is $K_m\in \mc{K}(f^n)$ such that $g^n(K_m)=K_m$. By compactness, there is a subsequence $\{K_{m_i}\}$ which converges in the Hausdorff topology, and it is easy to see that its limit must be a continuum $K$ which is essential in $E$, and $f^n(K)=K=g^n(K)$. Moreover, since $K$ is a Hausdorff limit of a sequence of disjoint frontiers, it follows that $K$ has empty interior. In fact, the sequence $K_{m_i}$ can be chosen so that it is either increasing or decreasing with respect to the ordering defined (using the notation of Definition \ref{def:frontier}) by $K_0<K_1$  if $K_0^-\subset K_1^-$ (and consequently $K_0^+\supset K_1^+$) for $K_0,K_1 \in \mc{K}(f^n)$; if the sequence is increasing then it is easy to see that $K = \bd\cup_i K_{m_i}^-$, which has empty interior (and similarly if the sequence is decreasing). 
Since $K$ is $f^n$-invariant and $g^n$-invariant and has empty interior, it contains a unique frontier $\hat{K}\in \mc{K}(f^n)\cap \mc{K}(g^n)$ (by Lemma \ref{lem:frontier} and its corollary). Thus, $g\notin \mc{R}^n_{f,E}$. This proves that $\diff^r_\omega(S)\sm \mc{R}^n_{f,E}$ is closed; hence $\mc{R}^n_{f,E}$ is open and dense.

Now consider a countable family $\mc{A}$ of closed annuli in $S$ such that for any closed annulus $E\subset S$ there is $E'\in \mc{A}$ such that $E$ is essential in the interior of $E'$. Such a family can be obtained as follows: Consider a sequence of (finite) triangulations $\{\mc{T}_i\}_{i\in \N}$ of $S$, such that the mesh of $\mc{T}_i$ tends to $0$ when $i\to \infty$. For each $i$, consider the family $\mc{F}_i$ of simple closed curves formed by sides of elements of $\mc{T}_i$. Note that every simple closed curve can be $C^0$-approximated by elements of $\mc{F}_i$. Let $\mc{A}$ be the family of all annuli whose boundaries are elements of $\mc{F}_i$ for some $i$.
That family is clearly countable. Furthermore, if $E\subset S$ is a closed annulus then we can consider an open annular neighborhood $A$ of $E$ so that $A\sm{E}$ is a union of two open annuli $A_1$ and $A_2$. If $i$ is large enough, $A_1$ and $A_2$ both contain some element of $\mc{F}_i$, one of which is essential in $A_1$ and the other in $A_2$. These curves bound an annulus $E'\in \mc{A}$, and $E$ essential in the interior of $E'$, so $\mc{A}$ has the required property.

We know that for each $E\in \mc{A}$ and $n>0$, the set $\mc{R}^n_{f,E}$ is open and dense; thus $$\mc{R}_f= \bigcap_{n\in \N,\, E\in \mc{A}} \mc{R}^n_{f,E}\cap \mathcal{G}_f $$
is a residual subset of $\diff^r_\omega(S)$. Let $g\in \mc{R}_f$ and let $K\subset S$ be a frontier such that $g^n(K)=K$ for some $n>0$. Note that, since $K$ is contained in an annulus, $K\subset E$ for some $E\in \mc{A}$. If $f^m(K)=K$ for some $m>0$, then $K$ is $f^{mn}$-periodic and $g^{mn}$-periodic, which is not possible because $g\in \mc{R}^{mn}_{f,E}$.
Thus no frontier in $S$ is periodic for both $f$ and $g$ simultaneously. By Lemma \ref{lem:transitive}, we conclude that $\IFS(f,g)$ is transitive for any $g\in \mc{R}_f$. This completes the proof.
\end{proof}

\bibliographystyle{amsalpha} 
\bibliography{tesis}

\providecommand{\bysame}{\leavevmode\hbox to3em{\hrulefill}\thinspace}
\providecommand{\MR}{\relax\ifhmode\unskip\space\fi MR }
\providecommand{\MRhref}[2]{%
  \href{http://www.ams.org/mathscinet-getitem?mr=#1}{#2}
}
\providecommand{\href}[2]{#2}
\begin{thebibliography}{DdlLS06}

\bibitem[ABC05]{abc}
M.-C. Arnaud, C.~Bonatti, and S.~Crovisier, \emph{Dynamiques symplectiques
  g\'en\'eriques}, Ergodic Theory Dynam. Systems \textbf{25} (2005), no.~5,
  1401--1436.

\bibitem[Arn64]{arnold}
V.~I. Arnol'd, \emph{Instability of dynamical systems with many degrees of
  freedom}, Dokl. Akad. Nauk SSSR \textbf{156} (1964), 9--12.

\bibitem[AZ05]{zanata}
S.~Addas-Zanata, \emph{Some extensions of the {P}oincar\'e-{B}irkhoff theorem
  to the cylinder and a remark on mappings of the torus homotopic to {D}ehn
  twists}, Nonlinearity \textbf{18} (2005), no.~5, 2243--2260.

\bibitem[BC04]{bonatti-crovisier}
C.~Bonatti and S.~Crovisier, \emph{R\'ecurrence et g\'en\'ericit\'e}, Invent.
  Math. \textbf{158} (2004), no.~1, 33--104.

\bibitem[CY04]{CY}
C.-Q. Cheng and J.~Yan, \emph{Existence of diffusion orbits in a priori
  unstable {H}amiltonian systems}, J. Differential Geom. \textbf{67} (2004),
  no.~3, 457--517.

\bibitem[DdlLS06]{DLS}
A.~Delshams, R.~de~la Llave, and T.~M. Seara, \emph{A gearmetric mechanism for
  diffusion in {H}amiltonian systems overcoming the large gap problem:
  heuristics and rigorous verification on a model}, Mem. Amer. Math. Soc.
  \textbf{179} (2006), no.~844, viii+141.

\bibitem[Dou92]{douady}
R.~Douady, \emph{Applications du th\'eor\`eme des tores invariants}, Ph.D.
  thesis, Universit\'e de Paris 7, 1992, Th\`ese de troisi\'eme cycle.

\bibitem[Kor09]{koro}
A.~Koropecki, \emph{Aperiodic invariant continua for surface homeomorphisms},
  Mathematische Zeitschrift (2009), published online 23 Jun 2009
  (doi:10.1007/s00209-009-0565-0).

\bibitem[LC07]{lecalvez-drift}
P.~Le~Calvez, \emph{Drift orbits for families of twist maps of the annulus},
  Ergodic Theory Dynam. Systems \textbf{27} (2007), no.~3, 869--879.

\bibitem[Mat81]{mather-area}
J.~Mather, \emph{Invariant subsets of area-preserving homeomorphisms of
  surfaces}, Advances in Math. Suppl. Studies \textbf{7B} (1981), 531--561.

\bibitem[Mat04]{Ma}
\bysame, \emph{Arnol'd diffusion. {I}. {A}nnouncement of results}, J. Math.
  Sci. \textbf{124} (2004), no.~5, 5275--5289.

\bibitem[Moe02]{moeckel}
R.~Moeckel, \emph{Generic drift on {C}antor sets of annuli}, Celestial
  mechanics ({E}vanston, {IL}, 1999), Contemp. Math., vol. 292, Amer. Math.
  Soc., Providence, RI, 2002, pp.~163--171.

\bibitem[Mos65]{moser}
J.~Moser, \emph{On the volume elements on a manifold}, Transactions of the
  American Mathematical Society \textbf{120} (1965), no.~2, 286--294.

\bibitem[MS02]{MS}
J.-P. Marco and D.~Sauzin, \emph{Stability and instability for {G}evrey
  quasi-convex near-integrable {H}amiltonian systems}, Publ. Math. Inst. Hautes
  \'Etudes Sci. (2002), no.~96, 199--275 (2003).

\bibitem[Nas06]{tesis-meysam}
M.~Nassiri, \emph{Robustly transitive sets in nearly integrable hamiltonian
  systems}, Ph.D. thesis, IMPA, 2006.

\bibitem[NP08]{np}
M.~Nassiri and E.~Pujals, \emph{Robust transitivity in hamiltonian dynamics},
  (preprint), 2008.

\bibitem[Oli87]{oliveira-torus}
F.~Oliveira, \emph{On the generic existence of homoclinic points}, Ergodic
  Theory Dynam. Systems \textbf{7} (1987), no.~4, 567--595.

\bibitem[Pix82]{pixton}
D.~Pixton, \emph{Planar homoclinic points}, J. Differential Equations
  \textbf{44} (1982), no.~3, 365--382.

\bibitem[PS04]{pugh-shub}
C.~Pugh and M.~Shub, \emph{Stable ergodicity}, Bull. Amer. Math. Soc. (N.S.)
  \textbf{41} (2004), no.~1, 1--41, With an appendix by Alexander Starkov.

\bibitem[Rob70]{robinson}
R.~C. Robinson, \emph{Generic properties of conservative systems}, American
  Journal of Mathematics \textbf{93} (1970), no.~3, 562--603.

\bibitem[SYC91]{embedology}
T.~Sauer, J.~Yorke, and M.~Casdagli, \emph{Embedology}, Journal of Statistical
  Physics \textbf{65} (1991), no.~3/4, 579--616.

\bibitem[Xia98]{xia2}
Z.~Xia, \emph{Arnold diffusion: a variational construction}, Proceedings of the
  {I}nternational {C}ongress of {M}athematicians, {V}ol. {II} ({B}erlin, 1998),
  no. Extra Vol. II, 1998, pp.~867--877 (electronic).

\bibitem[Xia06]{xia-area}
\bysame, \emph{Area-preserving surface diffeomorphisms}, Communications in
  Mathematical Physics \textbf{263} (2006), 723--735.

\bibitem[Zeh77]{zehnder}
E.~Zehnder, \emph{Note on smoothing symplectic and volume preserving
  diffeomorphisms}, Geometry and topology (Proc. III Latin. Amer. School of
  Math, IMPA, Rio de Janeiro, 1976), Springer, 1977, Lecture Notes in Math,
  Vol. 597, pp.~828--854.

\end{thebibliography}

\end{document}